\documentclass [12pt,a4paper,reqno]{amsart} \textwidth 165mm
 \input amssymb.sty
\newcommand{\ef}{\end{equation}}
\chardef\bslash=`\\ 

\newtheorem{thm}{Theorem} [section]
\newtheorem*{thm*}{Theorem}
\newtheorem{cor}[thm]{Corollary}
\newtheorem{lem}[thm]{Lemma}

 \newtheorem{defn}[thm]{Definition}

 \theoremstyle{remark}
\newtheorem{remark}[thm]{Remark}
\newtheorem{acknowledgment*}[thm] {Acknowledgment}

\newcommand{\thmref}[1]{Theorem~\ref{#1}}

\newcommand{\corref}[1]{Corollary~\ref{#1}}

 \renewcommand{\sectionmark}[1]{}

\newcommand{\doe}{\overset{\text{def}}{=}}
 
\newcommand{\loc} {\operatorname{loc}}

 \date{}
\begin{document}

\title[Correct Solvability
of the Sturm-Liouville Equation With Delayed Argument]{Correct Solvability
of the Sturm-Liouville Equation With Delayed Argument}
\author[N.A. Chernyavskaya]{N.A. Chernyavskaya}\address{Department of Mathematics, Ben-Gurion University of the
Negev, P.O.B. 653, Beer-Sheva, 84105, Israel}
\email{nina@math.bgu.ac.il}
\author[L.A. Shuster]{L.A. Shuster}
 \address{Department of Mathematics,
 Bar-Ilan University, 52900 Ramat Gan, Israel}
 \email{miriam@macs.biu.ac.il}

\begin{abstract}
We consider  the equation
\begin{equation}\label{1}
-y''(x)+q(x)y(x-\varphi(x))=f(x),\quad x\in \mathbb R
\end{equation}
where $f\in C(\mathbb R)$ and
\begin{equation}\label{2}
0\le \varphi \in C^{\loc}(R),\quad  1\le q\in C^{\loc}(\mathbb R).
\end{equation}
Here $C^{\loc}(\mathbb R)$ is the set of functions continuous in every point of the number
axis. By a solution of \eqref{1}, we mean any function $y,$ doubly continuously
differentiable everywhere in $\mathbb R,$ which satisfies \eqref{1}. We show that under
certain additional conditions on the functions $\varphi$ and $q$ to \eqref{2},
\eqref{1} has a unique solution $y,$ satisfying the inequality
$$\|y\|_{C(\mathbb R)}\le c\|f\|_{C(\mathbb R)}$$
where the constant $c\in(0,\infty)$ does not depend on the choice
of $f\in C(\mathbb R).$
\end{abstract}

\maketitle

\baselineskip 20pt

\section{Introduction}\label{introduction}
\setcounter{equation}{0} \numberwithin{equation}{section}

In the present paper, we consider the equation
\begin{equation}\label{1.1}
-y''(x)+q(x)y(x-\varphi(x))=f(x),\quad x\in \mathbb R
\end{equation}
where $f\in C(\mathbb R)$ and
\begin{equation}\label{1.2}
0\le \varphi \in C^{\loc}(R),\quad  1\le q\in C^{\loc}(\mathbb R).
\end{equation}

By the symbol $C^{\loc}(\mathbb R),$ we denote the set of functions continuous in every
point of the number axis $\mathbb R.$

By a solution of \eqref{1.1} we mean any doubly continuously differentiable function
$y(x)$, satisfying \eqref{1.1} for all $x\in\mathbb R.$ In addition, we say that
equation \eqref{1.1} is correctly solvable in $C(\mathbb R)$ if the following
assertions hold:
\begin{enumerate}
\item[I)] for every function $f\in C(\mathbb{R})$ equation \eqref{1.1} has a unique
solution $y\in C(\mathbb R);$
\item[II)] there is a constant $c\in(0,\infty)$ such that regardless of the choice
of $f\in C(\mathbb R),$ the solution $y\in C(\mathbb R)$ of \eqref{1.1} satisfies
the inequality
\begin{equation}\label{1.3}\|y\|_{C(\mathbb R)}\le c\|f\|_{C(\mathbb R)}.
\end{equation}
\end{enumerate}

Our goal is to study, in a conceptual way, the problem of correct
solvability of equation \eqref{1.1} in the space $C(\mathbb R)$.
(For brevity, below we say ``the question on I)--II", or ``the
problem I)--II)".) Such an investigation is needed because this is
the first time the problem I)--II) is being posed. Indeed, to the
best of our knowledge, for equations with delayed argument there
has been studied initial and boundary value problems on a
finite segment or on a semi-axis (see \cite{1,7,11,12,13}). However, the
special feature of problem I)--II) is that equation \eqref{1.1} is
considered on the whole axis, and requirements to its solutions
are imposed apart from I)--II). Therefore, the main result of the
paper is statement asserting that problem I)--II) makes sense,
i.e., the set of equations \eqref{1.1} correctly solvable in
$C(\mathbb R)$ is non-empty. This statement follows from the
following theorem which is our main result.

\begin{thm}\label{thm1.1}
If, together with \eqref{1.2}, the following two conditions hold:
\begin{enumerate}
\item[{\rm 1)}] there is a constant $a\ge 1$ such that for all
$x\in\mathbb R$ the following inequalities hold:
\begin{equation}\label{1.4}a^{-1}q(x)\le q(t)\le a q(x)\quad\text{for}\quad \forall t\in
[x-1,x+1]; \end{equation} \item[{\rm 2)}]
\begin{equation}\label{1.5}
\sigma\le1/6\sqrt
a,\quad\text{where}\quad\sigma\doe \sup_{x\in\mathbb R}(\varphi(x)q(x)),\end{equation}
\end{enumerate} then equation \eqref{1.1} is
correctly solvable in $C(\mathbb R).$ In addition, equation
\eqref{1.1} is separable in $C(\mathbb R),$ i.e., there is a
constant $c\in(0,\infty)$ such that regardless of the choice of
$f\in C(\mathbb R),$ the solution $y\in C(\mathbb R)$ of
\eqref{1.1} satisfies the inequality
\begin{equation}\label{1.6} \|y''(x)\|_{C(\mathbb
R)}+\|q(x)y(x-\varphi(x))\|_{C(\mathbb R)}\le c\|f(x)\|_{C(\mathbb R)}.
\end{equation}
\end{thm}

\begin{cor}\label{cor1.2}
There is a constant $c\in(0,\infty)$ such that the solution $y\in
C(\mathbb R)$ of \eqref{1.1} satisfies the
estimate\begin{equation}\label{1.7} \|q(x)y(x)\|_{C(\mathbb R)}\le
c\|f(x)\|_{C(\mathbb R)},\quad \forall f\in C(\mathbb R).
\end{equation}
\end{cor}

The paper is constructed as follows. In \S2, we collect the
information needed for the proofs; in \S3, we present some
auxiliary assertions; \S4 contains a proof of \thmref{thm1.1};
and, finally, in \S5 we give an example of this theorem.

The authors thank Prof. B. Sklyar for productive discussions and
useful remarks.

\section{Preliminaries}

The information presented below is used in the proofs. Here and
throughout the sequel, we assume that conditions \eqref{1.2} are
satisfies. They are not referred to and do not appear in the
statements.

\begin{thm}\label{thm2.1} \cite{6,4} There exists a fundamental
system of solutions (FSS) $\{u(x),v(x)\},$ $x\in\mathbb R$ of the
equation
\begin{equation}\label{2.1}
z''(x)=q(x)z(x),\qquad x\in\mathbb R
\end{equation}
which has the following properties:
\begin{equation}
\begin{aligned}\label{2.2}
u(x)>0,\quad &v(x)>0,\quad u'(x)<0,\quad v'(x)>0,\qquad
x\in\mathbb R,\\
&v'(x)u(x)-u'(x)v(x)=1,\qquad x\in\mathbb R
\end{aligned}
\end{equation}
 \begin{align*}
 \lim_{x\to-\infty} v(x)&=\lim_{x\to-\infty}v'(x)=\lim_{x\to\infty}
 u(x)=\lim_{x\to\infty}u'(x)=0\\
 \lim_{x\to\infty} v(x)&=\lim_{x\to\infty}v'(x)=\lim_{x\to-\infty}
 u(x)=\lim_{x\to-\infty}|u'(x)|=\infty.
\end{align*}
\begin{equation}\label{2.3}
|p'(x)|<1,\qquad x\in\mathbb R;\qquad \rho(x)=u(x)v(x),\qquad
x\in\mathbb R,\end{equation}
\begin{equation}\label{2.4}
\frac{|u'(x)|}{u(x)}=\frac{1-\rho'(x)}{2\rho(x)},\qquad
\frac{v'(x)}{v(x)}=\frac{1+\rho'(x)}{2\rho(x)},\qquad x\in\mathbb
R.\end{equation}
\end{thm}

\begin{lem}\label{lem2.2} \cite{3} For a given $x\in\mathbb R$
consider the following equations in $d\ge0:$
\begin{equation}\label{2.5}
\int_0^{\sqrt {2}d}\int_{x-t}^x q(\xi)d\xi
dt=1,\qquad\int_0^{\sqrt{2}d}\int_x^{x+t}q(\xi)d\xi dt=1.
\end{equation}
Each of equations \eqref{2.5} has a unique finite positive
solution.
\end{lem}

Denote by $d_1(x),$ $d_2(x)$, $x\in\mathbb R,$ the solutions of
equations \eqref{2.5}, respectively.

\begin{thm}\label{thm2.3} \cite{3} We have the inequalities
\begin{equation}\label{2.6}
\frac{1}{\sqrt2}\le \frac{|u'(x)|}{u(x)} d_2(x); \
\frac{v'(x)}{v(x)}d_1(x)\le \sqrt 2,\qquad x\in\mathbb R,
\end{equation}
\begin{equation}\label{2.7}
\frac{1}{\sqrt2}\ \frac{d_1(x)d_2(x)}{d_1(x)+d_2(x)}\le
\rho(x)\le\sqrt 2\frac{d_1(x)d_2(x)}{d_1(x)+d_2(x)},\qquad
x\in\mathbb R,
\end{equation}
\end{thm}

Consider the equation
\begin{equation}\label{2.8}
-y''(x)+q(x)y(x)=f(x),\qquad x\in\mathbb R.
\end{equation}
By a solution of \eqref{2.8} we mean any doubly continuously
differentiable function $y(x)$ satisfying \eqref{2.8} for all
$x\in\mathbb R.$

\begin{defn}\label{defn2.4} \cite{5} We say that equation
\eqref{2.8} is correctly solvable in $C(\mathbb R)$ if the
following conditions are satisfied:
\begin{enumerate}
\item[{\rm a)}] for every function $f\in C(\mathbb R)$ there is a
unique solution $y\in C(\mathbb R)$ of equation \eqref{2.8};
\item[{\rm b)}] there is a constant $c\in(0,\infty)$ such that
regardless of the choice of $f\in C(\mathbb R),$ the solution
$y\in C(\mathbb R)$ of \eqref{2.8} satisfies the estimate
\begin{equation}\label{2.9}
\|y\|_{C(\mathbb R)}\le c\|f\|_{C(\mathbb R)}.
\end{equation}
\end{enumerate}
\end{defn}

\begin{remark}\label{rem2.5}
By $c,$ $c(\cdot)$, we denote absolute positive constants which
are not essential for exposition and may differ even with a single
chain of calculations.
\end{remark}

\begin{thm}\label{thm2.6} \cite{5} Equation \eqref{2.8} is
correctly solvable in $C(\mathbb R).$ Its solution $y\in C(\mathbb
R)$ is of the form
\begin{equation}\label{2.10}
y(x)=(Gf)(x)=\int_{-\infty}^\infty G(x,t)f(t)dt,\qquad x\in\mathbb
R.
\end{equation}
Here $G(x,t)$ is the Green function of equation \eqref{2.8}:
\begin{equation}\label{2.11}
G(x,t)=\begin{cases} u(x)v(t),&\quad x\ge t\\
u(t)v(x),&\quad x\le t
\end{cases}
\end{equation}
\end{thm}

Denote
\begin{equation}\label{2.12}
\mathcal D(\mathbb R)=\{y\in C(\mathbb R):y\in
C_{\loc}^{(2)}(\mathbb R),\ -y''(x)+q(x)y(x)\in C(\mathbb R),\
x\in\mathbb R\},
\end{equation}
\begin{equation}\label{2.13}
(\mathcal L y)(x)=-y''(x)+q(x)y(x),\quad x\in\mathbb R,\quad
y\in\mathcal D(\mathbb R).
\end{equation}
Here $C_{\loc}^{(2)}(\mathbb R)$ is the set of functions doubly
continuously differentiable for $x\in\mathbb R.$

\begin{thm}\label{thm2.7} \cite{5}
The operator $\mathcal L:\mathcal D(\mathbb R)\to C(\mathbb R)$ is
continuously invertible. We have the equality (see \eqref{2.10}):
\begin{equation}\label{2.14}
\mathcal L^{-1}=G.
\end{equation}
\end{thm}

\begin{defn}\label{defn2.8} \cite{2} We say that equation
\eqref{2.8} is separable in $C(\mathbb R)$ if there is a constant
$c\in(0,\infty)$ such that regardless of the choice of $f\in
C(\mathbb R),$ the solution $y\in C(\mathbb R)$ of \eqref{2.8}
satisfies the inequality
\begin{equation}\label{2.15}
\|y''\|_{C(\mathbb R)}+\|qy\|_{C(\mathbb R)}\le c\|f\|_{C(\mathbb
R)}.
\end{equation}
\end{defn}

\begin{remark}\label{rem2.9}
The problem of separating the operator $\mathcal L$ into summands:
$$\|\mathcal Ly\|_{C(\mathbb R)}\le \|y''\|_{C(\mathbb
R)}+\|qy\|_{C(\mathbb R)}\le c\|\mathcal Ly\|_{C(\mathbb
R)},\quad\forall y\in\mathcal D(\mathbb R)$$ was first studied in
\cite{8,9} in the space $L_2(\mathbb R).$
\end{remark}

\begin{thm}\label{thm2.10} \cite[pp.~84--85]{10} Let $\varphi(x)$,
$x\in\mathbb R$ be a non-negative, continuous function, and let
$y(x),$ $x\in\mathbb R$ be a doubly continuously differentiable
function. Then we have the equality
\begin{equation}\label{2.16}
y(x)=y(x-\varphi(x))+\varphi(x)y'(x)-\int_{x-\varphi(x)}^xy''(\xi)(\xi-x+\varphi
(x))d\xi,\qquad x\in\mathbb R.
\end{equation}
\end{thm}

\section{Auxiliary Assertions}

Some of the statements presented below are interesting in their
own right.

\begin{lem}\label{lem3.1}
We have the inequalities (see \eqref{2.5})
\begin{equation}\label{3.1}
\sup_{x\in\mathbb R} d_1(x)\le 1,\qquad\sup_{x\in\mathbb
R}d_2(x)\le 1.
\end{equation}
\end{lem}

\begin{proof}
{} From \eqref{1.2} and \eqref{2.5} it follows that
$$1=\int_0^{\sqrt{2}d_1(x)}\int_{x-t}^x q(\xi)d\xi
dt\ge\int_0^{\sqrt {2}d_1(x)}\int_{x-t}^x 1d\xi \ge d_1^2(x),\qquad
x\in\mathbb R\quad\Rightarrow\quad \eqref{3.1}.
$$
The second inequality in \eqref{3.1} can be checked similarly.
\end{proof}

\begin{lem}\label{lem3.2} We have the inequalities
\begin{equation}\label{3.2}
\|\mathcal L^{-1}\|_{C(\mathbb R)\to C(\mathbb R)}\le 1,
\end{equation}
\begin{equation}\label{3.3}
\left\|\frac{d}{dx}\mathcal L^{-1}\right\|_{C(\mathbb R)\to
C(\mathbb R)}\le\sqrt 2.
\end{equation}
\end{lem}

 \begin{proof}
 In the following relations we only use \thmref{thm2.1}:
\begin{equation*}
 \left.\begin{array}{ll}
 v'(x)=v'(x)-v'(-\infty)=\int_{-\infty}^x
 v''(t)dt=\int_{-\infty}^x q(t)v(t)dt\\ \\
 -u'(x)=u'(\infty)-u'(x)=\int_x^\infty u''(t)dt=\int_x^\infty
 q(t)u(t)dt
 \end{array}\right\}\quad\Rightarrow
 \end{equation*}
 \begin{align*}
 1&=v'(x)u(x)-u'(x)v(x)=u(x)\int_{-\infty}^xq(t)v(t)dt+v(x)\int_x^\infty
 q(t)u(t)dt\\
 &=\int_{-\infty}^\infty q(t)G(x,t)dt\ge\int_{-\infty}^\infty
 G(x,t)dt.\end{align*}
This implies that (see \eqref{2.14})
$$\|\mathcal L^{-1}\|_{C(\mathbb R)\to C(\mathbb
R)}=\sup_{x\in\mathbb R}\int_{-\infty}^\infty G(x,t)dt\le
1\quad\Rightarrow\quad \eqref{3.2}.$$

In the following relations, together with \thmref{thm2.1}, we use
\eqref{2.6}, \eqref{2.14} and \eqref{3.1}:
\begin{align*}
\left\|\frac{d}{dx}\mathcal L^{-1}\right\|_{C(\mathbb R)\to
C(\mathbb R)}&=\sup_{x\in\mathbb
R}\left|\frac{d}{dx}\int_{-\infty}^\infty
G(x,t)dt\right|=\sup_{x\in\mathbb R}\left|u'(x)\int_{-\infty}^x
v(t)dt+v'(x)\int_x^\infty u(t)dt\right|\\
&\le \sup_{x\in\mathbb
R}\bigg[|u'(x)|\int_{-\infty}^x\frac{v(t)}{v'(t)}\cdot
v'(t)dt+v'(x)\int_x^\infty\frac{|u(t)|}{|u'(t)|}\cdot
|u'(t)|dt\bigg]\\
&\le\sqrt 2\sup_{x\in\mathbb
R}\left[|u'(x)|\int_{-\infty}^xd_1(t)v'(t)dt+v'(x)\int_x^\infty
d_2(t)(-u'(t))dt\right]\\
&\le \sqrt 2\sup_{x\in\mathbb
R}\left[|u'(x)|\int_{-\infty}^xv'(t)dt-v'(x)\int_x^\infty
u'(t)dt\right]\\
&=\sqrt 2\sup_{x\in\mathbb R}[v'(x)u(x)-u'(x)v(x)]=\sqrt 2.
\end{align*}
 \end{proof}

 \begin{lem} Under condition \eqref{1.4}, we have the inequalities
 \begin{equation}\label{3.4}
 \frac{1}{\sqrt {2a}}\ \frac{1}{\sqrt {q(x)}}\le d_1(x),
 d_2(x)\le\sqrt{2a}\ \frac{1}{\sqrt {q(x)}},\qquad x\in\mathbb R.
 \end{equation}
 \end{lem}

 \begin{proof}
 Estimates \eqref{3.4} for $d_1(x)$ and $d_2(x)$ are proved in the
 same way; therefore, here we only consider $d_1(x).$ For a given
 $x\in\mathbb R,$ consider the equation in $d\ge0:$
 \begin{equation}\label{3.5}
 F(d)=1,\qquad F(d)=d\cdot\int_{x-d}^x q(\xi)d\xi.
 \end{equation}
 Clearly, on $[0,\infty)$ the function $F(d)$ is monotone
 increasing, and $F(d)\ge d^2.$ Since $F(0)=0,$ $F(\infty)=\infty,$
 we conclude that \eqref{3.5} has a unique positive solution.

 Denote this solution by $\hat d(x).$ Clearly, $\hat d(x)\le 1$
 for $x\in\mathbb R$ because $$1=\hat d(x)\cdot\int_{x-\hat d(x)}^x
 q(\xi)d\xi\ge\hat d(x)\int_{x-\hat d(x)}^x 1 dt= \hat d^2(x).
 $$
 In addition, from the first mean value theorem and \eqref{1.4},
 it follows that
 \begin{align}
 1&=\hat d(x)\int_{x-\hat d(x)}^x q(t)dt=q(\tilde x)\hat
 d^2(x),\quad \tilde x\in [x-1,x]\quad\Rightarrow\nonumber\\
  \hat d(x)&=\frac{1}{\sqrt{q(\tilde x)}}
  =\sqrt{\frac{q(x)}{q(\tilde x)}}\ \frac{1}{\sqrt{q(x)}}\le
 \sqrt{\frac{a}{q(x)}},\quad x\in\mathbb R,\label{3.6}
 \end{align}
 \begin{equation}\label{3.7}
 \hat d(x)=\frac{1}{\sqrt{q(\tilde x)}}=\sqrt\frac{q(x)}{q(\tilde
 x)}\ \frac{1}{\sqrt{q(x)}}\ge \frac{1}{\sqrt{aq(x)}},\quad
 x\in\mathbb R.
 \end{equation}

 Now, from \eqref{2.5}, properties of $F(d),$ $d\ge0,$ and
 \eqref{3.6} and \eqref{3.7}, it follows that
\begin{gather*} 1=\int_0^{\sqrt 2 d_1(x)}\int_{x-t}^x q(\xi)d\xi dt\le \sqrt
2 d_1(x)\int_{x-\sqrt 2 d_1(x)}^x q(\xi)d\xi=F(\sqrt 2
d_1(x))\quad\Rightarrow\\
\sqrt 2 d_1(x)\ge \hat d(x)
\ge\frac{1}{\sqrt{aq(x)}}\quad\Rightarrow\quad
d_1(x)\ge\frac{1}{\sqrt{2aq(x)}},
\end{gather*}
\begin{align*}
1&=\int_0^{\sqrt 2 d_1(x)}\int_{x-t}^x
q(\xi)d\xi\ge\int_{\frac{d_1(x)}{\sqrt 2}}^{\sqrt 2
d_1(x)}\int_{x-t}^x q(\xi)d\xi dt\ge \frac{d_1(x)}{\sqrt
2}\int_{x-\frac{d_1(x)}{\sqrt 2}}^x q(\xi)d\xi\\
&=F\left(\frac{d_1(x)}{\sqrt
2}\right)\quad\Rightarrow\quad\frac{d_1(x)}{\sqrt 2}\le\hat
d(x)\le \frac{a}{\sqrt{q(x)}}\quad\Rightarrow\quad
d_1(x)\le\sqrt{\frac{2a}{q(x)}},\quad x\in\mathbb
R\quad\Rightarrow\quad\eqref{3.5}.
\end{align*}
 \end{proof}

 \begin{cor}\label{cor3.4} For $x\in\mathbb R$ we have the
 inequalities
 \begin{equation}\label{3.8}
 \frac{v(x)}{v(x-d_1(x))}\le\exp\big(2\sqrt 2 a^{3/2}\big);\qquad
 \frac{u(x)}{u(x+d_2(x))}\le \exp\big(2\sqrt 2a^{3/2}\big).
 \end{equation}
 \end{cor}

 \begin{proof}
 Below we consecutively use \eqref{2.6}, \eqref{3.4}, \eqref{3.1},
 \eqref{1.4} and once again \eqref{3.4}:
 \begin{align*}
 \ln\frac{v(x)}{v(x-d_1(x))}&=\int_{x-d_1(x)}^x\frac{v'(\xi)}{v(\xi)}d\xi\le
 \sqrt 2\int_{x-d_1(x)}^x\frac{d\xi}{d_1(\xi)}\\
 &\le 2\sqrt a\int_{x-d_1(x)}^x \sqrt{ q(\xi)}d\xi=2\sqrt
 a\int_{x-d_1(x)}^x\sqrt{\frac{q(\xi)}{q(x)}}\cdot\sqrt{q(x)}d\xi\\
 &\le 2a\sqrt{q(x)} d_1(x)\le 2\sqrt 2
 a^{3/2}\quad\Rightarrow\quad\eqref{3.8}.
 \end{align*}
 \end{proof}

 \begin{thm}\label{thm3.5}
 Let $r\in C^{\loc}(\mathbb R).$ Then we have the estimates
 \begin{equation}\label{3.9}
 \frac{1}{2\sqrt 2a}\exp\big(-2\sqrt 2 a^{3/2}\big)m_0(r,q)\le
 \|r\mathcal L^{-1}\|_{C(\mathbb R)\to C(\mathbb R)}\le 4a
 m_0(r,q),\end{equation}
 \begin{equation}\label{3.10}
  \|r\frac{d}{dx}\mathcal L^{-1}\|_{C(\mathbb R)\to C(\mathbb R)}\le
  8a^{3/2}m_1(r,q).\end{equation}
Here
\begin{equation}\label{3.11}
m_0(r,q)=\sup_{x\in\mathbb R}\frac{|r(x)|}{q(x)};\qquad
m_1(r,q)=\sup_{x\in\mathbb R}\frac{r(x)}{\sqrt{q(x)}}.
\end{equation}
In particular, equation \eqref{2.8} is separable in $C(\mathbb R)$, and we
have the inequalities
\begin{equation}\label{3.12}
\frac{1}{2\sqrt 2a}\exp\big(-2\sqrt 2a^{3/2}\big)\le\|q\mathcal
L^{-1}\big)\le\|_{C(\mathbb R)\to C(\mathbb R)}\le 4a,
\end{equation}
\begin{equation}\label{3.13}
\left\|\frac{d^2}{dx^2}\mathcal L^{-1}\right\|_{C(\mathbb R)\to
C(\mathbb R)}\le 4a+1.
\end{equation}
\end{thm}

\begin{proof}
We need the following useful assertion.

\begin{lem}\label{lem3.6}
For $x\in\mathbb R$ we have the inequality
\begin{equation}\label{3.14}
\int_{\infty}^\infty G(x,t)dt\le 4\int_{x-1}^{x+1}G(x;t)dt.
\end{equation}
\end{lem}

\begin{proof}
Let us check the relations
\begin{equation}\label{3.15}
\int_{-\infty}^x  v(t)dt\le 4\int_{x-1}^x v(t)dt,\qquad \int_x^\infty u(t)dt\le 4\int_x^{x+1}u(t)dt.
\end{equation}
These inequalities are proved in the same way; therefore,  below we only consider the second one. Denote
$$z(t)=e^{-t},\quad t\in\mathbb R.$$

The following relations are deduced from \thmref{thm2.1}:
\begin{align}
&\qquad\qquad\qquad\qquad u''(\xi)=q(\xi)u(\xi),\qquad z''(\xi)=z(\xi),\qquad \xi\in\mathbb R \nonumber\\
&\Rightarrow\quad[u'(\xi)z(\xi)-z'(\xi)u(\xi)]'=u''(\xi)z(\xi)-z''(\xi)u(\xi)=(q(\xi)-1)u(\xi)z(\xi)\ge0\nonumber\\
&\Rightarrow\quad-[u'(t)z(t)-z'(t)u(t)]=\int_t^\infty[u'(\xi)z(\xi)-z'(\xi)u(\xi)]'d\xi\ge0,\quad t\in\mathbb R\nonumber\\
 &\Rightarrow\quad u'(t)z(t)-z'(t)u(t)\le 0,\quad t\in\mathbb R\quad\Rightarrow\quad u'(t)\le -u(t),  \quad t\in\mathbb R\nonumber\\
 &\Rightarrow\quad \qquad\qquad\qquad\qquad\frac{u(t)}{u(x)}\le e^{x-t}\qquad\text{as}\qquad t\ge x,\quad x\in\mathbb R.\label{3.16}
\end{align}

Let $x\in \mathbb R,$ $x_n=x+n,$ $n=1,2,\dots\, .$ Below we use \thmref{thm2.1} and \eqref{3.16}
\begin{align*}
\int_x^\infty u(t)dt&=\int_x^{x_1}u(t)dt+\sum_{n=1}^\infty\int_{x_n}^{x_{n+1}}u(t)dt\\
&=\int_x^{x_1}u(t)dt\left[1+\sum_{k=1}^\infty\left(\int_{x_n}^{x_{n+1}}u(t)dt\right)\left(\int_x^{x_1}u(t) dt\right)^{-1}\right]\\
&\le\int_{x_1}^x u(t)dt\cdot\left[1+\sum_{n=1}^\infty\frac{u(x_n)}{u(x_1)}\right]\le\int_x^{x_1}u(t)dt\cdot\left[1+\sum_{n=1}
^\infty e^{-(n-1)}\right]\\
&=\int_x^{x_1}u(t)dt\left[1+\frac{1}{1-e^{-1}}\right]\le 4\int_x^{x_1}u(t)dt\quad\Rightarrow\quad\eqref{3.15}.
\end{align*}

Let us now go to \eqref{3.12}. Below we use \thmref{thm2.1}, \eqref{2.5}, \eqref{2.14}, \eqref{2.7}, \eqref{3.8} and \eqref{3.4}:
\begin{align*}
\|r\mathcal L^{-1}\|_{C(\mathbb R)\to C(\mathbb R)}&=\sup_{x\in\mathbb R}|r(x)|\int_{-\infty}^\infty G(x,t)dt\ge\sup_{x\in\mathbb R}|r(x)|\int_{x-d_1(x)}^{x+d_2(x)}G(,t)dt\\
&=\sup_{x\in\mathbb R}|r(x)|\left[u(x)\int_{x-d_1(x)}^x v(t)dt+v(x)\int_x^{x+d_2(x)}u(t)dt\right]\\
&\ge\sup_{x\in\mathbb R} |r(x)|\left[u(x)v(x-d_1(x))d_1(x)+v(x)u(x+d_2(x))d_2(x)\right]\\
&=\sup_{x\in\mathbb R}|r(x)|\rho(x)\left[\frac{v(x-d_1(x))}{v(x)}d_1(x)+\frac{u(x+d_2(x))}{u(x)}d_2(x)\right]\\
&\ge exp\big(-2\sqrt 2 a^{3/2}\big)\sup_{x\in\mathbb R}|r(x)|\rho(x)(d_1(x)+d_2(x))\\
&\ge \frac{1}{\sqrt 2}\exp\big(-2\sqrt 2 a^{3/2}\big)\sup_{x\in\mathbb R}|r(x)|d_1(x)d_2(x)\\
&\ge\frac{1}{2\sqrt 2 a}\exp\big(-2\sqrt 2 a^{3/2}\big)\sup_{x\in\mathbb R}\frac{|r(x)|}{q(x)}=\frac{\exp\big(-2\sqrt 2a^{3/2}\big)}{2\sqrt 2 a}m_0(r,q).
\end{align*}

Below, in the proof of the upper estimate in \eqref{3.9}, we use \thmref{thm2.1}, \eqref{2.14}, \eqref{3.15} and \eqref{1.4}:
\begin{align*}
\|r\mathcal L^{-1}\|_{C(\mathbb R)\to C(\mathbb R)}&=\sup_{x\in\mathbb R}|r(x)|\int_{-\infty}^\infty G(x,t)dt\\ &\le 4\sup_{x\in\mathbb R}\left(|r(x)|\int_{x-1}^{x+1}G(x,t)dt\right)
 =4\sup_{x\in\mathbb R}\frac{|r(x)|}{q(x)}\int_{x-1}^{x+1}\frac{q(x)}{q(t)}\cdot q(t)G(x,t)dt\\
 &\le 4a\sup_{x\in\mathbb R}\frac{|r(x)|}{q(x)}\int_{x-1}^{x+1}q(t)G(x,t)dt
 \le 4a\sup_{x\in\mathbb R}\frac{|r(x)|}{q(x)}\int_{-\infty}^\infty q(t)G(x,t)dt\\
 &=4a\sup_{x\in\mathbb R}\left[u(x)\int_{-\infty}^x v''(t)dt+v(x)\int_x^\infty u''(t)dt\right]\\
&=4a\sup_{x\in\mathbb R}\frac{|r(x)|}{q(x)}(v'(x)u(x) -u'(x)v(x))=4am_0(r,q).
\end{align*}

To prove \eqref{3.11}, we consecutively use \thmref{thm2.1}, \eqref{2.14}, \eqref{3.15}, \eqref{1.4}, \eqref{2.3}, \eqref{2.4}, \eqref{2.7} and \eqref{3.4}:
\begin{align*}
\left\|r\frac{d}{dx}\mathcal L^{-1}\right\|_{C(\mathbb R)\to C(\mathbb R)}&=\sup_{x\in\mathbb R}|r(x)|\left|\frac{d}{dx}\int_{-\infty}^\infty G(x,t)dt\right|\\
&\le 4\sup_{x\in\mathbb R}|r(x)|\left[|u'(x)|\int_{x-1}^x v(t)dt+v'(x)\int_x^{x+1}u(t)dt\right]\\
&=4\sup_{x\in\mathbb R}\frac{|r(x)|}{q(x)}\left[|u'(x)|\int_{x-1}^xv''(t)\frac{q(x)}{q(t)}q(t)
dt+v'(x)\int_x^{x+1}u''(t)\frac{q(x)}{q(t)}q(t)dt\right]\\
&\le 4a \sup_{x\in\mathbb R}\frac{|r(x)|}{q(x)}\left[|u'(x)|\int_{-\infty}^xv''(t)dt+v'(x)\int_x^{x+1}u''(t)dt\right]\\
&\le 8a\sup_{x\in\mathbb R}\left(\frac{|r(x)|}{q(x)}\cdot\frac{v'(x)}{v(x)}\cdot\frac{|u'(x)|}{u(x)}\cdot\rho(x)\right)=8a\sup_{x\in
\mathbb R}\left(\frac{|r(x)|}{q(x)}\cdot \frac{1-\rho{'}^{2}(x)}{4\rho(x)}\right)\\
&\le 2a\sup_{x\in\mathbb R}\left(\frac{|r(x)|}{q(x)}\frac{1}{\rho(x)}\right)\le 2\sqrt 2a\sup_{x\in\mathbb R}\frac{|r(x)|}{q(x)}\left(\frac{1}{d_1(x)}+\frac{1}{d_2(x)}\right)\\
&\le 8a ^{3/2}\sup_{x\in\mathbb R}\frac{|r(x)|}{\sqrt{q(x)}}=8a^{3/2}m_1(r,q).
\end{align*}
The proof of \eqref{3.12} is obvious, and \eqref{3.13} follows from \eqref{3.12}, \eqref{2.8} and the triangle inequality.
\end{proof}

Consider the system of equations
\begin{equation}
 \left.\begin{array}{ll}
z_1(x)=f_1(x)+(B_{11}z_1)(x)+(B_{22}z_2)(x)\\ \\
 z_2(x)=f_2(x)+(B_{21}z_1)(x)+(B_{22}z_2)(x)
 \end{array}\right\},\qquad x\in\mathbb R,\label{3.17}
 \end{equation}
 where $f_k\in C(\mathbb R),$ $k=1,2,$ $B_{ij}:C(\mathbb R)\to C(\mathbb R),$ $i,j=1,2$ are linear operators.

 \begin{lem}\label{lem3.7} Suppose that we have the inequality
 \begin{equation}\label{3.18}
 \|B_{ij}\|_{C(\mathbb R)\to C(\mathbb R)}\le\frac{1}{4},\qquad i,j=1,2.
 \end{equation}
 Then the system \eqref{3.17} has a unique solution $\{z_1,z_2\}$ such that $z_k\in C(\mathbb R),$ $k=1,2,$ and
 \begin{equation}\label{3.19}
 \|z_1\|_{C(\mathbb R)}+\|z_2\|_{C(\mathbb R)}\le 2\big(\|f_1\|_{C(\mathbb R)}+\|f_2\|_{C(\mathbb R)}\big).
 \end{equation}
 \end{lem}

 \begin{proof} Let us write down \eqref{3.12} in vector form. Set
 \begin{equation}\label{3.20}
 z(x):=\begin{cases} z_1(x)\\ z_2(x)\end{cases};\qquad f(x):=\begin{cases} f_1(x)\\ f_2(x)\end{cases};\qquad B:=\begin{pmatrix} B_{11} & B_{12}\\ B_{21} &B_{22}\end{pmatrix}.
 \end{equation}
Then the system \eqref{3.17} becomes
\begin{equation}\label{3.21}
z(x)=f(x)+(Bz)(x),\qquad x\in\mathbb R.
\end{equation}
Denote by $C_2(\mathbb R)$ the vector space of vector functions $z(x), $ $x\in\mathbb R$ with continuous congruents $z_k(x),$ $x\in\mathbb R,$ $k=1,2$ (see \eqref{3.20}), equipped with the norm
\begin{equation}\label{3.22}
\|z\|_{C_2(\mathbb R)}=\|z_1\|_{C(\mathbb R)}+\|z_2\|_{C(\mathbb R)}.
\end{equation}

Let us estimate the norm of the operator $B:C_2(\mathbb R)\to C_2(\mathbb R):$
\begin{align*}
\|Bz\|_{C_2(\mathbb R)}&=\|B_{11}z_1+B_{12}z_2\|_{C(\mathbb R)}+\|B_{21}z_1+B_{22}z_2\|_{C(\mathbb R)}\\
&\le \big(\|B_{11}z_1\|_{C(\mathbb R)}+\|B_{12}z_2\|_{C(\mathbb R)}\big)+ \big(\|B_{21}z_1\|_{C(\mathbb R)}+  \|B_{22}z_2\|_{C(\mathbb R)}\big)\\
&\le \frac{1}{2}\big (\|z_1\|_{C(\mathbb R)}+\|z_2\|_{C(\mathbb R)}\big)=\frac{1}{2}\|z\|_{C_2(\mathbb R)}.
\end{align*}

Therefore, the operator $B:C_2(\mathbb R)\to C_2(\mathbb R)$ is a compressing operator, and the lemma is proved.
 \end{proof}

\section {Proof of the Main Result}

Below we prove \thmref{thm1.1}. Let us introduce an operator $A:C(\mathbb R)\to C(\mathbb R)$ by the formula
\begin{equation}\label{4.1}
(Af)(x)\doe q(x)\int_{x-\varphi(x)}^x f(\xi)(\xi-x+\varphi(x))d\xi,\qquad x\in\mathbb R,\quad f\in C(\mathbb R).\end{equation}

\begin{lem}\label{lem4.1}
We have the inequalities:
\begin{equation}\label{4.2}
\|A\|_{C(\mathbb R)\to C(\mathbb R)}\le\frac{1}{36a},
\end{equation}
\begin{equation}\label{4.3}
\|(E-A)^{-1}\|_{C(\mathbb R)\to C(\mathbb R)}\le\frac{36}{35}.
\end{equation}
\end{lem}

\begin{proof}
Let $\in C(\mathbb R).$ Then we have the relations
\begin{align*}
\|Af\|_{C(\mathbb R)}&=\sup_{x\in\mathbb R}q(x)\left|\int_{x-\varphi(x)}^xf(\xi)(\xi-x+\varphi(x))d\xi\right|\\
&\sup_{x\in\mathbb R}\left[q(x)\int_{x-\varphi(x)}^x(\xi-x+\varphi(x))d\xi\right]\|f\|_{C(\mathbb R)}\le\sup_{x\in\mathbb R}\left[\frac{(q(x)\varphi(x))^2}{q(x)}\right]\cdot \|f\|_{C(\mathbb R)}\\
&\le \sup_{x\in\mathbb R}(q(x)\varphi(x))^2\cdot\|f\|_{C(\mathbb R)}\le \frac{1}{36a}\|f\|_{C(\mathbb R)}\quad\Rightarrow\quad\eqref{4.2}.
\end{align*}
Inequality \eqref{4.3} follows from \eqref{4.2} and the expansion of the operator $(E-A)^{-1}$ in powers of the operator $A.$\end{proof}

Let us introduce some more notation:
\begin{equation}\label{4.4}
f(x),\quad z_1(x),\quad z_2(x) \quad   \text{--\quad are functions from}\quad C(\mathbb R),
\end{equation}
\begin{equation}\label{4.5}
g(x)=[(E-A)^{-1}f](x),\qquad x\in\mathbb R,
\end{equation}
\begin{equation}\label{4.6}
(F_1z_1)(x)=\sum_{n=1}^\infty(A^n(qz_1))(x),\qquad x\in\mathbb R,
\end{equation}
\begin{equation}\label{4.7}
(F_2z_2)(x)=[(E-A)^{-1}(q\varphi z_2)](x),\qquad x\in\mathbb R.
\end{equation}

\begin{lem}\label{lem4.2}
We have the relations
\begin{equation}\label{4.8}
g(x)\in C(\mathbb R),\qquad F_1z_1\in C(\mathbb R),\qquad F_2z_2\in C(\mathbb R),
\end{equation}
\begin{equation}\label{4.9}
\|g\|_{C(\mathbb R)}\le\frac{36}{35}\|f\|_{ C(\mathbb R)},
\end{equation}
\begin{equation}\label{4.10}
\| F_1z_1\|_{C(\mathbb R)}\le\frac{\|z_1\|_{C(\mathbb R)}}{70},\qquad \| F_2z_3\|_{C(\mathbb R)}\le \frac{6}{35}\|z_2\|_{ C(\mathbb R)}.
\end{equation}
\end{lem}

\begin{proof}
Inclusions \eqref{4.8} follow from \eqref{4.1}, \eqref{4.2}, \eqref{4.3} and estimates \eqref{4.9} and \eqref{4.10}. Estimate \eqref{4.9} follows from \eqref{4.3}. Consider \eqref{4.10}. We have
\begin{align*}
\|Aqz_1\|_{C(\mathbb R)}&=\sup_{x\in\mathbb R}q(x)\left|\int_{x-\varphi(x)}^xq(\xi)(\xi-x+\varphi(x))z_1(\xi)d\xi\right|\\
 &\le \sup_{x\in\mathbb R}  q(x)^2\left[\int_{x-\varphi(x)}^x\frac{q(\xi)}{q(x)}(\xi-x+\varphi(x))d\xi\right]\cdot\|z_1\|_{C(\mathbb R)}\\
&\le \frac{a}{2}\sup_{x\in\mathbb R}(q(x)\varphi(x))^2\cdot\|z_1\|_{C(\mathbb R)}\le\frac{1}{72}\|z_1\|_{C(\mathbb R)}\quad\Rightarrow\ \text{(see \eqref{4.3})}:\\
\|F_1z_1\|_{C(\mathbb R)}&\le\sum_{n=1}^\infty \|A^n(qz_1\|_{C(\mathbb R)}\le\sum_{n=1}^\infty\|A\|_{C(\mathbb R)\to C(\mathbb R)}^{n-1}\cdot\|A(qz_1)\|_{C(\mathbb R)}\\
&\le \left(1-\frac{1}{36a}\right)^{-1}\cdot\frac{\|z_1\|}{72}\le\frac{36}{35}\frac{\|z_1\|_{C(\mathbb R)} }{72}=\frac{\|z_1\|}{70}\quad\Rightarrow\eqref{4.10}.
\end{align*}
Similarly,
\begin{align*}
\|F_2z_2\|_{C(\mathbb R)}&=\sup_{x\in\mathbb R}\|(E-A)^{-1}(q\varphi z_2)\|_{C(\mathbb R)}\\
 &\le\|(E-A)^{-1}\|_{C(\mathbb R)\to C(\mathbb R)}  \cdot\frac{\|z_2\|_{C(\mathbb R)}}{6\sqrt{a}}\le \frac{6}{35}\|z_2\|_{C(\mathbb R)}\quad\Rightarrow\quad \eqref{4.10}.
\end{align*}

Consider the system \eqref{3.17}, where we set (see \eqref{4.1}, \eqref{2.10}, \eqref{2.14}, \eqref{4.4}, \eqref{4.5}, \eqref{4.6} and \eqref{4.7}):
\begin{equation}\label{4.11}
f_1(x)=(Gg)(x),\qquad f_2(x)=\frac{d}{dx}(Gg)(x),\qquad x\in R,
\end{equation}
\begin{equation}\label{4.12}
(B_{11}z_1)(x) =-(GF_1z_1)(x),\qquad x\in R,\\
\end{equation}
\begin{equation}\label{4.13}
(B_{12}z_2)(x) =-(GF_2z_2)(x),\qquad x\in R,\\
\end{equation}
\begin{equation}\label{4.14}
(B_{21}z_1)(x) =-\left(\frac{d}{dx}GF_1z_1\right)(x),\qquad x\in R,\\
\end{equation}
\begin{equation}\label{4.15}
(B_{22}z_2)(x) =-\left(\frac{d}{dx}GF_2z_2 \right)(x),\qquad x\in R.\\
\end{equation}
\end{proof}

\begin{lem}\label{lem4.3}
We have the following estimates for the norms of the operator $B_{ij},$ $i,j=1,2$ (see \eqref{4.12}, \eqref{4.13}, \eqref{4.14} and \eqref{4.15}):
 \begin{equation}\label{4.16}
\|B_{ij}\|_{C(\mathbb R)\to C(\mathbb R)}\le\frac{1}{4},\qquad i,j=1,2.
\end{equation}
\end{lem}

\begin{proof}
The assertion of the lemma follows from \eqref{2.14}, \eqref{3.2},  \eqref{3.3}, \eqref{4.2}, \eqref{4.3} and \eqref{4.10}. For example, for $i=j=1$ and $i=j=2,$ respectively, we have
\begin{align*}
\|B_{11}z_1\|_{C(\mathbb R)}&=\|GF_1z_1\|_{C(\mathbb R)}\le\|G\|_{C(\mathbb R)\to C(\mathbb R)}\cdot \|F_1z_1\|_{C(\mathbb R)}\le \frac{1}{70}\|z_1\|_{C(\mathbb R)}\quad\Rightarrow \eqref{4.16};\\
\|B_{22}z_2\|_{C(\mathbb R)}&=\left\|\frac{d}{dx}GF_2z_2\right\|_{C(\mathbb R)}\le\left\|\frac{d}{dx}G\right\|_{C(\mathbb R)\to C(\mathbb R)}\cdot\|F_2z_2\|_{C(\mathbb R)}\\
&\le \frac{6\sqrt 2}{35}\|z_2\|_{C(\mathbb R)}<\frac{1}{4}\|z_2\|_{C(\mathbb R)}\quad\Rightarrow\eqref{4.16}.\end{align*}
\end{proof}

{}From Lemmas \ref{lem4.3} and \ref{lem3.7} it follows that the system of equations
\begin{equation}\label{4.17}
\begin{cases}
z_1(x)=[G(E-A)^{-1}f](x)-[GF_1z_1](x)+[GF_2z_2](x)\\
&,\quad x\in\mathbb R\\
z_2(x)=\left[\frac{d}{dx}G(E-A)^{-1}f\right](x)-\left[\frac{d}{dx}GF_1z_1\right](x)+\left[\frac{d}{dx}
GF_2z_2\right](x)\end{cases}
\end{equation}
has a unique solution $z(x)\in C_2(\mathbb R) $ (see \eqref{3.20}), and (see \eqref{4.6}, \eqref{3.2}, \eqref{3.3},\eqref{3.19}, \eqref{4.4} and \eqref{4.11}), we have
\begin{equation}\label{4.18}
\|z_1\|_{C(\mathbb R)}+\|z_2\|_{C(\mathbb R)}\le \left[\|G(E-A)^{-1}f\|_{C(\mathbb R)}+\left\|\frac{d}{dx} G(E-A)^{-1}f\right\|_{C(\mathbb R)}\right]\le c\|f\|_{C(\mathbb R)}.
\end{equation}

Note that there is a relationship between the functions $z_1(x)$ and $z_2(x)$ that can be checked in a straightforward way (see \eqref{4.17}):
$$z_2(x)=z_1'(x),\qquad x\in\mathbb R.$$
Denote
\begin{equation}\label{4.19}
y(x)=z_1(x),\quad x\in\mathbb R\quad \Rightarrow\quad y'(x)=z_1'(x)=z_2(x),\quad x\in\mathbb R.
\end{equation}
By \eqref{4.19}, the first equation in \eqref{4.17} and the estimate \eqref{4.18} take the form \eqref{4.20} and \eqref{4.21}, respectively:
\begin{equation}\label{4.20}
y(x)=(G(E-A)^{-1}f)(x)-(GF_1y)(x)+(GF_2y)(x),\qquad x\in\mathbb R,
\end{equation}
\begin{equation}\label{4.21}
\|y\|_{C(\mathbb R)}+\|y'\|_{C(\mathbb R)}\le c\|f\|_{C(\mathbb R)}.
\end{equation}
{}From \eqref{4.20} it follows that
\begin{equation}\label{4.22}
y(x)=(Gw)(x),\quad x\in\mathbb R;\qquad w(x)=((E-A)^{-1}f)(x)-(F_1y')(x)+(F_2y)(x),
\end{equation}
and (see \eqref{4.9}, \eqref{3.2}, \eqref{3.3}, \eqref{4.21}, we have
\begin{equation}\label{4.23}
\|w\|_{C(\mathbb R)}\le c\|f\|_{C(\mathbb R)}.
\end{equation}

Thus, by Theorems \ref{thm2.6}, \ref{thm2.7} and \ref{thm3.5}, \eqref{4.22} and \eqref{4.23}, we get
\begin{equation}\label{4.24}
-y''(x)+q(x)y(x)=w(x)=((E-A)^{-1}f)(x)-(F_1y)(x)+(F_2y')(x),\quad x\in\mathbb R,
\end{equation}
\begin{equation}\label{4.25}
\|y'' \|_{C(\mathbb R)}+\|qy\|_{C(\mathbb R)}\le c\|w\|_{C(\mathbb R)}\le c\|f\|_{C(\mathbb R)}.
\end{equation}
Since $qy\in C(\mathbb R)$ (see \eqref{4.21} and \eqref{4.25}) and
$$((E-A)^{-1}qy)(x)=q(x)y(x)+(F_1y)(x),\qquad x\in\mathbb R,$$
we obtain, by combining the last equality with \eqref{4.24}, that
$$-y''(x)=(E-A)^{-1}[f(x)-q(x)y(x)+q(x)\varphi(x)y'(x)],\qquad x\in\mathbb R.$$
But $y''\in C(\mathbb R), $
and therefore
\begin{gather}
-y''(x)+(Ay'')(x)=f(x)-q(x)y(x)+q(x)\varphi(x)y'(x),\quad x\in\mathbb R\quad\Rightarrow\nonumber\\
-y''(x)=q(x)\left[y(x)-\varphi(x)y'(x)+\int_{x-\varphi(x)}^xy''(\xi)(\xi-x+\varphi(x)d\xi\right]=f(x),\ \ x\in\mathbb R.\label{4.26}
\end{gather}

{}From \eqref{4.26} and \eqref{2.16} we obtain \eqref{1.1}, i.e., $y(x),$ $x\in\mathbb R,$ is a solution of \eqref{1.1}, and we have the estimate \eqref{1.3} (see \eqref{4.21}). The uniqueness of such a solution \eqref{1.1} follows from the linearity of this equation and the estimate \eqref{1.3}.
{}From \eqref{4.25} and the triangle inequality, we get \eqref{1.5}.
\end{proof}

\begin{proof}[Proof of \corref{cor1.2}]
The estimate \eqref{1.7} follows from \thmref{thm1.1} and \eqref{4.25}.\end{proof}

\section{Example}
Below we consider equation \eqref{1.1} with
\begin{equation}\label{5.1}
q(x)=2(1+x^2)+(1+x^2)\sin(|x|^2),\qquad x\in\mathbb R.
\end{equation}
The function \eqref{5.1} satisfies \eqref{1.2}, and with the help of \thmref{thm1.1}, we show that such an equation \eqref{1.1} is correctly solvable in the space $C(\mathbb R)$ if $\sigma:=\frac{1}{31}$ (see \eqref{1.5}). To prove this fact, we use the following simple lemma, which can be useful for checking condition \eqref{1.4}.

\begin{lem}\label{lem5.1} Suppose that we are given a function $q(x),$ $x\in\mathbb R$, and \eqref{1.2} holds. If there is a positive, continuously differentiable function $q_1(x)$ for $x\in\mathbb R $ such that
\begin{enumerate}
\item[1)] for all $x\in\mathbb R,$ we have the inequalities
\begin{equation}\label{5.2}
 \nu^{-1}q_1(x)\le q(x)\le \nu q_1(x)
\end{equation}
where the constant $\nu\in[1,\infty)$ does not depend on the choice of a point $x\in \mathbb R$;
\item[2)] $s<\infty$ where
\begin{equation}\label{5.3}
s=\sup_{x\in\mathbb R}\frac{|q_1'(x)|}{q_1(x)}.
\end{equation}
Then the function $q$ satisfies condition \eqref{1.4} for $a=\nu^2e^s.$
\end{enumerate}
\end{lem}

\begin{proof}
Let $t\in[x-1,x+1],$ $x\in\mathbb R.$ In the following relations, we use \eqref{5.2} and \eqref{5.3}:
\begin{align*}
 \left|
 \ln\frac{q_1(t)}{q_1(t)}\right|&=|\ln q_1(t)-\ln q_1(x)|
  =\left|\int_x^t\frac{q_1'(\xi)}{q_1(\xi)}d\xi
  \right|\\
&\le\left|\int_x^t\frac{|q_1'(\xi)|}{q_1(\xi)}d\xi\right|\le s|x-t|\le s\qquad\Rightarrow
\end{align*}
\begin{gather*}
 \qquad\qquad\qquad\qquad e^{-s}\le\frac{q_1(t)}{q_1(x)}\le e^s,\qquad\qquad \quad |t-x|\le 1\qquad\Rightarrow\\ \frac{q(t)}{q(x)}=\frac{q(t)}{q_1(t)}\cdot\frac{q_1(t)}{q_1(x)}\cdot\frac{q_1(x)}{q(x)}\le\nu^2e^s,\qquad |t-x|\le 1,\\
\frac{q(t)}{q(x)}=\frac{q(t)}{q_1(t)}\cdot\frac{q_1(t)}{q_1(x)}\cdot\frac{q_1(x)}{q(x)}\ge\frac{1}
{\nu^2}e^{-s},\qquad |t-x|\le 1.
\end{gather*}
In the case \eqref{5.1}, clearly, $q_1(x)=1+x^2,$ $x\in\mathbb R$, because
$$\frac{1+x^2}{3}<1+x^2\le 2(1+x^2)+(1+x^2)\sin(|x|^2)\le 3(1+x^2),\qquad x\in\mathbb R,$$
and therefore $\nu=3.$ In addition, $s=1$ because
$$s=\sup_{x\in\mathbb R}\frac{|q_1'(x)|}{q_1(x)}=\sup_{x\in\mathbb R}\frac{2|x|}{1+x^2}\le 1.$$
Hence
$$\frac{1}{6\sqrt a}=\frac{1}{6}\cdot\frac{1}{3\sqrt e}=\frac{1}{18\sqrt e}\ge \frac{1}{31}=\sigma,$$
as required.
\end{proof}

\end{document}